\documentclass[letterpaper, 10 pt, conference]{ieeeconf}  

\usepackage{amsmath}
\usepackage{amssymb}
\usepackage{mathrsfs}
\usepackage[all]{xy}
\newcommand{\ds}{\displaystyle}

\def\qed{\ifvmode\removelastskip\fi
{\unskip\nobreak\hfil\penalty50\hbox{}\nobreak\hfil \hbox{\vrule
height1.2ex width1.2ex}\parfillskip=0pt \finalhyphendemerits=0
\par \smallskip}}

\usepackage{fancybox}

\renewcommand{\theequation}{\thesection.\arabic{equation}}
\newtheorem{thm}{Theorem}[section]
\newtheorem{prop}[thm]{Proposition}
\newtheorem{defin}[thm]{Definition}

\newtheorem{corol}[thm]{Corollary}

\IEEEoverridecommandlockouts                              
\title{\LARGE \bf
Characterization of accessibility for affine connection control
systems at some points with nonzero velocity}

\author{Mar\'ia Barbero Li\~n\'an 
\thanks{This work has been partially supported by MICINN (Spain) Grants MTM2008–00689 and MTM2009–08166; 2009SGR1338 of
the Catalan government, IRSES project GEOMECH (246981) within the 7th European Community Framework Program,
by Beatriu de Pin\'os fellowship from Comissionat per a Universitats i Recerca del Departament d'Innovaci\'o, Universitats i
Empresa of Generalitat de Catalunya and by Juan de la Cierva fellowship from MICINN.}
\thanks{This work was mainly developed while a postdoctoral fellow at Queen's University, Kingston, ON, Canada. M. Barbero Li\~n\'an is with the Institute for the Mathematical Sciences (CSIC-UAM-UC3M-UCM), Madrid, Spain
        {\tt\small mbarbero@icmat.es}}%
}

\begin{document}

\maketitle
\thispagestyle{empty}
\pagestyle{empty}

\begin{abstract}

Affine connection control systems are mechanical control systems
that model a wide range of real systems such as robotic legs,
hovercrafts, planar rigid bodies, rolling pennies, snakeboards and so
on. In 1997 the accessibility and a particular notion of
controllability was intrinsically described by A. D. Lewis and R.
Murray at points of zero velocity. Here, we present a novel
generalization of the description of accessibility algebra for those
systems at some points with nonzero velocity as long as the affine connection restricts to
the distribution given by the symmetric closure. The results are used to describe the
accessibility algebra of different mechanical control systems.
\end{abstract}


\section{Introduction}

A wide range of mechanical control systems can be described in terms
of an affine connection on the configuration manifold
$Q$~\cite{AndrewBook}\@.

An \textbf{analytic affine connection control system (ACCS)
$\Sigma$} is a 4-tuple $(Q,\nabla,\mathscr{Y},U)$ where
\begin{itemize}
\item $Q$ is a configuration manifold,
\item $\nabla$ is an affine connection on $Q$,
\item $\mathscr{Y}=\{Y_1,\ldots,Y_r\}$ is a set of vector fields on
$Q$ so-called \textbf{control or input vector fields},
\item $U$ is a subset of $\mathbb{R}^r$. 
\end{itemize}

The typical assumption on $U$ is to be an almost proper control set,
that is, zero is in the convex hull of $U$ ($0\in {\rm conv}(U)$)
and the affine hull of $U$ is the entire Euclidean space
$\mathbb{R}^r$ (${\rm aff}(U)=\mathbb{R}^r$).

The dynamics are described by a second-order differential equation on
$Q$
\begin{equation}\nabla_{\gamma'(t)}\gamma'(t)=u^aY_a(\gamma(t))\label{Eq-nabla} \end{equation} where
$\gamma\colon I\subset \mathbb{R} \rightarrow Q$, $u\colon I
\rightarrow U\subset \mathbb{R}^r$.

It is well-known the control system~\eqref{Eq-nabla} is equivalent to a
first\textendash{}order control\textendash{}affine system on $TQ$
\begin{equation}
\Upsilon'(t)=Z(\Upsilon(t))+u^aY^V_a(\Upsilon(t)),\label{Eq-Control-Affine-Equiv}
\end{equation}
where $\Upsilon \colon I \rightarrow TQ$ projects onto $Q$ to
$\gamma$ through $\tau_Q\colon TQ \rightarrow Q$, $Z$ is the
geodesic spray and $Y^V_a$ is the vertical lift of the vector field
$Y_a$ on $Q$ to $TQ$.

One of the key objects to describe the accessibility and
controllability of control-affine systems such as~\eqref{Eq-Control-Affine-Equiv} is the involutive closure ${\rm Lie}^{\infty}(Z, \mathscr{Y}^V)$
of the vector fields in~\eqref{Eq-Control-Affine-Equiv} and the ``symmetric" closure ${\rm Sym}^{(\infty)}(\mathscr{Y})$ of 
$\mathscr{Y}$ that will be defined in Section~\ref{Sec:ZeroV}. In~\cite{LM97} the accessibility distribution was characterized
at points of zero velocity.
In this paper we are going to extend this characterization to points
with velocity in ${\rm Sym}^{(\infty)}({\mathscr Y})$  as long as
$\nabla$ is restricted to ${\rm Sym}^{(\infty)}({\mathscr Y})$, cf.
Proposition~\ref{Corol-AccDistrNonZero}. This will allow us to
extend the characterization of accessibility for mechanical control
systems in~\cite{LM97}, cf. Theorem~\ref{Thm-AccDistrNonZero}.

The paper is organized as follows. Section~\ref{Sec:Notations}
reviews the main properties and related results about accessibility
of the mechanical control systems under study. It also introduces the
notion of a connection which restricts to a distribution that will
be an essential assumption for the results below.
Sections~\ref{SubSec:PrimitiveBrack} and~\ref{Sec:Description} contain all the novel results of this
paper. Theorem~\ref{Prop-Type-VF} characterizes a family of
generators of ${\rm Lie}^{(\infty)}(Z,\mathscr{Y}^V)$ at any $v_q\in
TQ$. This family is not a minimal set of generators, but it provides
a suitable set of generators for obtaining all the remaining results
in this paper that culminate in
Propositions~\ref{Corol-Zinv},~\ref{Corol-AccDistrNonZero} and
Theorem~\ref{Thm-AccDistrNonZero}. Section~\ref{Sec:Examples}
applies the above results to two mechanical control systems: the
planar rigid body with variable-direction thruster and the rolling
penny.

We shall assume all manifolds are paracompact, Hausdorff, and of class
$\mathcal{C}^\omega$ (analytic).  All maps and geometric objects will be assumed to be
of class $\mathcal{C}^\omega$\@.  The set of analytic
functions on a manifold $Q$ is denoted by $\mathcal{C}^\omega(Q)$\@.  For a
manifold $Q$\@, its tangent bundle will be denoted by $\tau_Q\colon
TQ\rightarrow Q$\@. If $\pi\colon E\rightarrow M$ is a vector bundle over $M$\@,
we denote by $\Gamma^\omega(E)$ the set of analytic sections of $E$\@.  By
$0_x\in E_x$ we denote the zero vector in the fiber at $x$\@. The analyticity assumption is necessary to have sufficient
and necessary characterizations of the accessibility algebra, cf.
Theorems~\ref{Th:SussmaanEtAl},~\ref{Th:LM97},~\ref{Thm-AccDistrNonZero}\@. The control set is always assumed to be almost proper. 

\section{Notation and previous concepts}\label{Sec:Notations}

This section contains the properties of control systems and some notions of
differential geometry that would be necessary at some point in the
development of this paper. This is not a detailed description. We refer 
to~\cite{AndrewBook,KN,98Lewis,LM97,NvdS, SJ72} for more
details. 

\subsection{Properties of a control system}\label{Sec:properties}

Typical studied properties of a control system are accessibility and
controllability. The former has been widely understood \cite{NvdS,SJ72}, whereas the
latter is still under study \cite{AndrewBook}. To define them formally we first need
to introduce the notion of a \textbf{reachable set ${\mathcal
R}(v_q,T)$ from $v_q\in TQ$ at time $T\in \mathbb{R}^+$}
\begin{equation*}{\mathcal R}(v_q,T)=\left\{w\in TQ \; \left| \; \begin{array}{l} \exists \mbox{ a trajectory }(\Upsilon, u) 
\mbox{ of }~\eqref{Eq-Control-Affine-Equiv} \\
\mbox{ such that } \Upsilon(0)=v_q, \\ \;
\Upsilon(T)=w\end{array}\right.\right\}\end{equation*} Then the
\textbf{reachable set from $v_q$ up to time $T$} is defined as
follows
\begin{equation*}{\mathcal R}(v_q,\leq T)=\bigcup_{t\in[0,T]}{\mathcal
R}(v_q,t).\end{equation*}

These reachable sets collect the information about states and
velocities that can be reached by trajectories of the control
system~\eqref{Eq-Control-Affine-Equiv}\@. When we study a mechanical system, we might just be
interested in the reachable states. That is why we define the
\textbf{reachable set ${\mathcal R}_Q(q,T)$ in $Q$ from $q$ at
time $T$},
\begin{equation*}{\mathcal R}_Q(q,T)=\left\{\overline{q}\in Q \; \left| \; \begin{array}{l} \exists \mbox{ a trajectory }
(\Upsilon, u) \mbox{ of }~\eqref{Eq-Control-Affine-Equiv}\\
\mbox{ such that } \tau_Q (\Upsilon(0))=q, \\ \; \tau_Q
(\Upsilon(T))=\overline{q}\end{array}\right.\right\}.\end{equation*} As before,
the \textbf{reachable set ${\mathcal R}_Q(q,\leq T)$ in $Q$ from $v_q$
up to time $T$} is defined as follows
\begin{equation*}{\mathcal R}_Q(q,\leq T)=\bigcup_{t\in[0,T]}{\mathcal
R}_Q(q,t).\end{equation*}

\begin{defin} Let $\Sigma$ be an analytic affine
connection control system and let $v_q\in TQ$.

 \begin{enumerate}
 \item $\Sigma$ is
\textbf{accessible from $v_q$} if there exists $T>0$ such that
\begin{equation*}{\rm int}({\mathcal R}(v_q,\leq t))\neq \emptyset \quad
\mbox{for } t\in (0,T].\end{equation*}

\item $\Sigma$ is \textbf{configuration accessible from
$q\in Q$} if there exists
$T>0$ such that \begin{equation*}{\rm int} ({\mathcal R}_Q(q,\leq
t)) \neq \emptyset \quad \mbox{for each } t\in (0,T].\end{equation*}

\item $\Sigma$ is \textbf{small-time locally controllable
(STLC) from $v_q$} if there exists $T>0$ such that
\begin{equation*} v_q\in {\rm int}({\mathcal R}(v_q,\leq t))\quad
\mbox{for each } t\in (0,T].\end{equation*}

\item $\Sigma$ is \textbf{small-time locally configuration controllable
(STLCC) from $v_q$} if there exists $T>0$ such that
\begin{equation*} v_q\in {\rm int}({\mathcal R}_Q(v_q,\leq t))\quad
\mbox{for each } t\in (0,T].\end{equation*}
\end{enumerate} \label{Def:AccControl}
\end{defin}

In this paper we only focus on extending the characterization of the
accessibility algebra given in~\cite{LM97}, but we include the
definition of controllability for completeness of the paper.

\subsection{Characterization of accessibility algebra at zero velocity}\label{Sec:ZeroV}

For any control-affine system on $Q$
\begin{equation} \dot{q}=f_0(q)+u^af_a(q),
\label{eq:control-aff-system}
\end{equation}
the \textbf{accessibility distribution at $q$} is the involutive
closure at $q$ of the distribution generated by the drift vector
field $f_0$ and the control vector fields $\{f_a\}_{a=1,\dots,r}$.

\begin{thm}[\cite{NvdS, SJ72}] The control-affine system
in~\eqref{eq:control-aff-system} is accessible from $q$ if and only
if ${\rm Lie}^{(\infty)}(f_0,f_1,\dots,
f_r)_q\simeq T_qQ$.\label{Th:SussmaanEtAl}
\end{thm}

This same result can also be applied to the control system~\eqref{Eq-Control-Affine-Equiv}\@. Thus, 
$\Sigma$ is accessible from $v_q\in TQ$ if and only if
${\rm Lie}^{(\infty)}(Z,\mathscr{Y}^V)_{v_q}\simeq T_{v_q}TQ$.
Observe that to define the accessibility algebra the control set
does not play any role.

Let us review now the intrinsic description of the accessibility
algebra at zero velocity given in~\cite{LM97}\@.  At $0_q$ there exists the
following natural isomorphism \begin{equation}T_{0_q}TQ \simeq T_qQ
\oplus T_qQ,\label{Eq:Split-0q}\end{equation}
cf.~\cite[Lemma~6.33]{AndrewBook}.

Given an affine connection $\nabla$\@, there exists a complementary
subbundle $HTQ$ of the vertical subbundle $VTQ=\ker(T\tau_Q)$ such
that $TTQ=HTQ\oplus VTQ$\@. This complementary subbundle is called
\textbf{horizontal subbundle}\@.  We refer to \cite[page~87]{KN} for
a characterization of this subbundle. Then the splitting
in~\eqref{Eq:Split-0q} corresponds with the splitting into
horizontal and vertical subspace.

Remember that for an
affine connection the \textbf{symmetric product} of two vector
fields is defined  as follows
\begin{equation*} \langle X \colon Y \rangle= \nabla_X Y +\nabla_Y
X.
\end{equation*}
Analogously to the involutive closure of vector fields,
we define the symmetric closure ${\rm Sym}^{(\infty)} ({\mathscr Y})$  of a set $\mathscr{Y}$ of control
vector fields as 
the distribution being the smallest $\mathbb{R}$-subspace of vector
fields containing $\mathscr{Y}$ and closed under the symmetric
product.

\begin{thm}{\cite[Lemma 5.7]{LM97}} Let $\Sigma$ be an ACCS. If ${\mathcal C}_\Sigma=\{Z,\mathscr{Y}^V\}$, then, for $q\in Q$,
\begin{equation*}{\rm Lie}^{(\infty)}({\mathcal C}_\Sigma)_{0_q}={\rm
Lie}^{(\infty)}({\rm Sym}^{(\infty)} ({\mathscr Y}))_{q}\oplus {\rm
Sym}^{(\infty)}({\mathscr Y})_{q}.
\end{equation*}
Moreover: \begin{enumerate} \item $\Sigma$ is accessible from $0_q$ if
and only if ${\rm Sym}^{(\infty)}(\mathscr{Y})_{q}=
T_{q}Q$. 
\item $\Sigma$ is configuration accessible from $0_q$ if and only if
${\rm Lie}^{(\infty)}({\rm
Sym}^{(\infty)}(\mathscr{Y}))_q=T_qQ$.
\end{enumerate}\label{Th:LM97}
\end{thm}

\subsection{On connections restricted to distributions} \label{Sec:NablaRestricted}

As mentioned above, the generalization of Theorem~\ref{Th:LM97} in this paper is only possible under suitable assumptions,
cf. Theorem~\ref{Thm-AccDistrNonZero}\@.
An \textbf{affine connection $\nabla$ restricts to a distribution}
$\mathcal{D}$ if $\nabla_X Y$ is a section of
$\mathcal{D}$ for every section $Y$ of $\mathcal{D}$ and $X\in \Gamma^\omega(TQ)$~\cite{98Lewis}.

Associated with an affine connection, we define the
\textbf{curvature of $\nabla$} as a $(1,3)$-tensor field on $Q$ such
that
\begin{equation*}R(X,Y)W=\nabla_X\nabla_YW-\nabla_Y\nabla_XW-\nabla_{[X,Y]}W\end{equation*}
for $X,Y,W\in \Gamma^\omega(TQ)$.
\begin{prop}{\cite[Proposition 4.3]{98Lewis}} Let $q\in Q$ and let
$u,v\in T_qQ$. If $\nabla$ restricts to ${\mathcal D}$, then the
endomorphism $R(u, v)$ of $T_qQ$ leaves the subspace $\mathcal{D}_q$
invariant. \label{Prop:Andrew98R}
\end{prop}

Another object of interest is the following one.

\begin{defin}
A distribution $\mathcal{D}$ on $Q$ is \textbf{geodesically
invariant} under an affine connection $\nabla$ on $Q$ if, for every
geodesic $\gamma\colon I\to Q$ for which
$\gamma'(t_0)\in\mathcal{D}_{\gamma(t_0)}$ for some $t_0\in I$\@, it
holds that $\gamma'(t)\in\mathcal{D}_{\gamma(t)}$ for every $t\in
I$\@.
\end{defin}

From the definition we have that a distribution $\mathcal{D}$ on $Q$
is geodesically invariant under an affine connection
$\nabla$ on $Q$ if, as a submanifold of $TQ$\@, $\mathcal{D}$ is
invariant under the geodesic spray associated with $\nabla$\@.  One
can then show that a distribution is geodesically invariant if and
only if the symmetric product of any $\mathcal{D}$-valued vector
fields is again a $\mathcal{D}$-valued vector field~\cite[Theorem
5.4]{98Lewis}.

Observe that if $\nabla$ restricts to a distribution
$\mathcal{D}$, then $\mathcal{D}$ is geodesically invariant under
$\nabla$.
%
%

\section{Primitive brackets}\label{SubSec:PrimitiveBrack}

A systematic way to construct the accessibility algebra is to
compute iteratively the Lie brackets in ${\rm Lie}^{\infty}(Z,
{\mathscr{Y}}^V)$. One of the key points to obtain the characterization of the
accessibility algebra reviewed in Theorem~\ref{Th:LM97} was to
identify a family of primitive brackets that were enough to span the
accessibility distribution at $0_q$. The novel result of this
section is to identify the primitive brackets for the accessibility
distribution at non-zero velocity. At zero velocity we recover the
same primitive brackets as in~\cite{LM97}.

\begin{thm} \label{Prop-Type-VF} If $v_q\in TQ$, then any tangent
vector in ${\rm Lie}^{\infty}(Z, {\mathscr{Y}}^V)_{v_q}$ is spanned
by a linear combination of elements in:
\begin{enumerate}
\item \label{VF1} $\{Z_{v_q}\}$;
\item \label{VF2} ${\mathcal A}_{v_q}=({\rm Sym}^{(\infty)}(\mathscr{Y}))^V_{v_q}$, i. e., the vertical lift of the
smallest distribution containing ${\mathscr{Y}}$ and closed under
the symmetric product;
\item \label{VF3} ${\mathcal B}_{v_q}=\{{\rm ad}_Z^l({\rm Sym}^{(\infty)}(\mathscr{Y}))^V_{v_q} \, | \, l\in \mathbb{N} \cup
\{0\}\}$, i.e., the smallest distribution containing ${\mathcal A}$
and invariant under the geodesic spray;
\item \label{VF4} ${\mathcal C}_{v_q}={\rm Lie}^{(\infty)}(\{{\rm ad}_Z^l({\rm Sym}^{(\infty)}(\mathscr{Y}))^V \, | \, l\in \mathbb{N}\})_{v_q}$, i.e., the
smallest involutive distribution containing ${\mathcal B}-{\mathcal A}$.
\end{enumerate}
\begin{proof}
As stated in \cite[Proposition 3.1]{LM97} any vector field in ${\rm
Lie}^{(\infty)}(Z, {\mathscr{Y}}^V)$ can be written as a linear combination of 
brackets of the form \[[W_k,[W_{k-1},[\dots,[W_2,W_1]\dots]]]\]
where $W_i\in \{Z, {\mathscr{Y}}^V\}$, 
$i=1,\dots,k$. Thus we only need to prove this proposition for
Lie brackets of this form.

Let us prove the result by induction on the length $k$ of the Lie
brackets that generate the entire ${\rm Lie}^{(\infty)}(Z,
{\mathscr{Y}}^V)$ at $v_q$.

For $k=1$, the only possible vector fields are $Z$ and $Y^V$. These
vector fields live in the distributions of type \ref{VF1} and
\ref{VF2}, respectively.

We prove now the result for some Lie brackets of small length so
that it is clear how to prove the induction step.

For $k=2$, the only nonzero vector field is $[Z,Y^V]$ that lives in
the distribution of type \ref{VF3}.

For $k=3$, there are $[Z,[Z,Y^V]]$ and $[Y_2^V,[Z,Y_1^V]]=\langle
Y_2\colon Y_1 \rangle^V$ that live in the distribution of type
\ref{VF3} and \ref{VF2}, respectively.

For $k=4$, the vector fields are the following ones:
\begin{itemize}
\item $[Z,[Z,[Z,Y^V]]]={\rm ad}_Z^3 (Y^V)$ that lives in the
distribution of type \ref{VF3}.
\item $[Y_2^V,[Z,[Z,Y_1^V]]]$ is not that straightforward and we must use the Jacobi identity.
\begin{align*}
[Y_2^V,&\, [Z,[Z,Y_1^V]]]=-[[Z,Y_1^V],[Y_2^V,Z]  \\&\,
-[Z,[[Z,Y_1^V],Y_2^V]]=[[Z,Y_1^V],[Z,Y_2^V]] 
\\&\, +[Z,\langle Y_2\colon Y_1 \rangle^V].
\end{align*}
Thus the first summand is one of type \ref{VF4} and the second one
is of type \ref{VF3}.
\item $[Z,[Y_2^V,[Z,Y_1^V]]]=[Z,\langle Y_2\colon Y_1 \rangle^V]$ is of type \ref{VF3}.
\item $[Y_3^V,[Y_2^V,[Z,Y_1^V]]]=[Y_3^V,\langle Y_2 \colon Y_1 \rangle^V]=0$.
\end{itemize}

Now we have already given an idea how to rewrite the Lie brackets
conveniently in order to express them as a linear combination of
vector fields in any of the four types of vector fields described in
the statement of Theorem~\ref{Prop-Type-VF}. Let us prove the
induction step.

Assume that any Lie bracket up to length $k$ can be written as a
linear combination of vector fields among the four types mentioned
in Theorem~\ref{Prop-Type-VF}, then let us
prove it is true for Lie brackets of length $k+1$.

By assumption any Lie bracket of length $k>1$ is either of type
\ref{VF2} or  \ref{VF3} or \ref{VF4}. Then we only need to bracket
these types of vector fields with $Z$ and the family of vertical vector
fields ${\mathscr{Y}^V}$.

\begin{itemize}
\item $[Z,{\mathcal A}]$ is of type \ref{VF3}.
\item $[Z,{\mathcal B}]$ is still of type \ref{VF3}.
\item $[Z,{\mathcal C}]$ needs some more work.

A Lie bracket in ${\mathcal C}$ can be written as
\[[{\rm ad}_Z^{k_s}W_s^V,[{\rm ad}_Z^{k_{s-1}}W_{s-1}^V,[\dots,[{\rm ad}_Z^{k_2}W_2^V,{\rm ad}_Z^{k_1}W_1^V]\dots]]],\]
where $W_i\in {\rm Sym}^{(\infty)}(\mathscr{Y})$, $k_i\in \mathbb{N}$, $i=1,\dots,s$.

We prove that $[Z,{\mathcal C}]$ can always be written as a linear
combination of vector fields of type 3 and 4 by induction on the
number of vertical vector fields involved in a vector field of type
${\mathcal C}$.

For $s=1$,
\[[Z,{\rm ad}_Z^{k_1}W_1^V]={\rm ad}_Z^{k_1+1}W_1^V\]
is trivially of type \ref{VF3}.

Let us introduce the following notation to shorten the expressions:
\begin{align}{\mathcal V}_i&={\rm ad}_Z^{k_i}W_i^V, \label{Eq-V-1} \\
{\mathcal V}_{ij}&=[{\rm ad}_Z^{k_j}W_j^V,{\rm ad}_Z^{k_i}W_i^V]=[{\mathcal V}_j,{\mathcal V}_i], \label{Eq-V-2} \\
&\dots \nonumber \\
{\mathcal V}_{i_1\dots i_s}&=[{\rm
ad}_Z^{k_{i_s}}W_{i_s}^V,[\dots,[{\rm ad}_Z^{k_{i_2}}W_{i_2}^V,{\rm
ad}_Z^{i_1}W_{i_1}^V]\dots]]. \label{Eq-V-s}
\end{align}

For $s=2$,
\begin{align*}[Z,{\mathcal V}_{12}]&=-[{\mathcal V}_1,[Z,{\mathcal V}_2]]-[{\mathcal V}_2,[{\mathcal V}_1,Z]]\\& =[{\rm
ad}_Z^{k_2+1}W_2^V,{\mathcal V}_1]+[{\mathcal V}_2,{\rm
ad}_Z^{k_1+1}W_1^V].
\end{align*}
Thus it is of type \ref{VF4}.

 The induction step is the following one for $s>2$: If
$[Z,{\mathcal V}_{i_1\dots i_s}]$ can be written as a linear
combination of vector fields of type 4 for any vector field
${\mathcal V}_{i_1\dots i_s}$ in ${\mathcal C}$ with $s$ vertical
vector fields involved, see (\ref{Eq-V-s}), then $[Z,{\mathcal
V}_{i_1\dots i_{s+1}}]$ is a linear combination of vector fields of
type 4 for any vector field ${\mathcal V}_{i_1\dots i_{s+1}}$ in
${\mathcal C}$ with $s+1$ vertical vector fields involved.
\begin{multline*}[Z,{\mathcal V}_{i_1\dots i_{s+1}}]
=[{\rm ad}_Z^{k_{s+1}+1}W_{s+1}^V,{\mathcal V}_{i_1\dots
i_{s}}]\\+[{\rm ad}_Z^{k_{s+1}}W_{s+1}^V,[Z,{\mathcal V}_{i_1\dots
i_{s}}]].
\end{multline*}
The first summand is of type \ref{VF4}. By induction hypothesis
$[Z,{\mathcal V}_{i_1\dots
i_{s}}]$ can be expressed as a linear
combination of vector fields of type \ref{VF4}. The induction concludes.

\item $[Y^V, {\mathcal A}]=0$.

\item $[Y^V,{\mathcal B}]$ needs some work. An element of ${\mathcal
B}$ is given by ${\rm ad}^s_Z(W^V)$ where $W^V\in {\mathcal A}$,
then
\begin{align*}
[Y^V,&\, {\rm ad}^s_Z(W^V)]=\\&\, -[{\rm
ad}^{s-1}_Z(W^V),[Y^V,Z]]-[Z,[{\rm ad}^{s-1}_Z(W^V),Y^V]]\\&\,
=[{\rm ad}^{s-1}_Z(W^V),[Z,Y^V]]+[Z,[Y^V,{\rm
ad}^{s-1}_Z(W^V)]].\end{align*} The first summand is of type
\ref{VF4}. For the second one we need to do further study. By induction on $s$ 
$[Y^V,{\rm ad}^{s-1}_Z(W^V)]$ can
be written as a linear combination of vector fields of type \ref{VF2},
\ref{VF3} and \ref{VF4}. As the bracketing of $Z$ with any of these
three types of distributions is of type \ref{VF3} and \ref{VF4}
according to the first cases considered, our reasoning concludes.

\item $[Y^V, {\mathcal C}]$ needs some more work. As in the case $[Z,{\mathcal C}]$, we consider
induction on the number $s$ of vertical lift of vector fields in
${\mathcal C}$ in order to prove that all the vector fields in
$[Y^V, {\mathcal C}]$ can be expressed as a linear combination of
vector fields of type \ref{VF4}.

For $s=1$, we have $[Y^V,{\rm ad}_Z^{k_1}W_1^V]$.

We prove it by induction on $k_1$. For $k_1=1$,
\[[Y^V,{\rm
ad}_Z^{1}W_1^V]=\langle Y\colon W_1\rangle^V.\]

By induction we prove that for $k_1>1$ \begin{align*}[Y^V,& \, {\rm
ad}_Z^{k_1}W_1^V]=(k_1-1)[{\rm ad}_Z^{k_1-1}(W_1^V),{\rm
ad}_Z^1(Y^V)]\\& \, +\sum_{j=1}^{k_1-2}A_j^{(k_1)}[{\rm
ad}_Z^j(W_1^V),{\rm ad}_Z^{k_1-j}(Y^V)]\\ & \, +{\rm
ad}_Z^{k_1-1}(\langle Y\colon W_1 \rangle^V)\\& \,
=\sum_{j=1}^{k_1-1}A_j^{(k_1)}[{\rm ad}_Z^j(W_1^V),{\rm
ad}_Z^{k_1-j}(Y^V)]\\& \, +{\rm ad}_Z^{k_1-1}(\langle Y\colon W_1
\rangle^V),\end{align*} where $A_j^{(k_1)}\in \mathbb{N}$ satisfies
\begin{align*}
A^{(k_1)}_0&=\,1,\\
A^{(k_1)}_1&=\,A^{(k_1-1)}_1,\\
A^{(k_1)}_j&=\,A_j^{(k_1-1)}+A_{j-1}^{(k_1-1)}, \quad \rm{for} \;\; 2\leq j \leq k_1-2, \\
A^{(k_1)}_{k_1-1}&=\,k_1-1.
\end{align*}

For $k_1=2$, \begin{align*}[Y^V,& \; {\rm
ad}_Z^{2}W_1^V]=\\&-[[Z,W_1^V],[Y^V,Z]]-[Z,[[Z,W_1^V],Y^V]]\\&=[[Z,W_1^V],[Z,Y^V]]\\&+[Z,\langle
Y\colon W_1\rangle^V].\end{align*}

Assume it is true for $k_1$, let us prove it for $k_1+1$.
\begin{align*}[Y^V, & \, {\rm
ad}_Z^{k_1+1}W_1^V]=-[{\rm ad}_Z^{k_1}W_1^V,[Y^V,Z]]\\ & \,
 -[Z,[{\rm ad}_Z^{k_1}W_1^V,Y^V]]=[{\rm
ad}_Z^{k_1}W_1^V,[Z,Y^V]]\\& \, +[Z,[(k_1-1)[{\rm
ad}_Z^{k_1-1}(W_1^V),{\rm ad}_Z^1(Y^V)]]]\\
& \, +\sum_{j=1}^{k_1-2}A^{(k_1)}_j[Z,[{\rm ad}_Z^j(W_1^V),{\rm
ad}_Z^{k_1-j}(Y^V)]]\\ & \, +{\rm ad}_Z^{k_1}(\langle Y\colon W_1
\rangle^V)=[{\rm ad}_Z^{k_1}W_1^V,[Z,Y^V]]\\& \, -(k_1-1)[{\rm
ad}_Z^1(Y^V),[Z,{\rm ad}_Z^{k_1-1}(W_1^V)]]\\& \, -(k_1-1)[{\rm
ad}_Z^{k_1-1}(W_1^V),[{\rm ad}_Z^1(Y^V),Z]]\\&
\,-\sum_{j=1}^{k_1-2}A^{(k_1)}_j[{\rm ad}_Z^{k_1-j}(Y^V),[Z,{\rm
ad}_Z^j(W_1^V)]]\\& \,-\sum_{j=1}^{k_1-2}A^{(k_1)}_j[{\rm
ad}_Z^j(W_1^V),[{\rm ad}_Z^{k_1-j}(Y^V),Z]]\\& \,+{\rm
ad}_Z^{k_1}(\langle Y\colon W_1 \rangle^V)=k_1[{\rm
ad}_Z^{k_1}W_1^V,[Z,Y^V]]\\& \,+A^{(k_1+1)}_j\sum_{j=1}^{k_1-1}[{\rm
ad}_Z^{j}(W_1^V),{\rm ad}_Z^{k_1-j+1}(Y^V)]\\ & \, +{\rm
ad}_Z^{k_1}(\langle Y\colon W_1 \rangle^V)
\end{align*}
Thus it is a linear combination of vector of type \ref{VF3} and
\ref{VF4}.

For $s=2$, according to the notation introduced in (\ref{Eq-V-1})
and (\ref{Eq-V-2}):
\begin{align*}[Y^V,{\mathcal V}_{12}]&=-[{\mathcal V}_1,[Y^V,{\mathcal V}_2]]-[{\mathcal V}_2,[{\mathcal V}_1,Y^V]]
\\&=-[{\mathcal V}_1,[Y^V,{\mathcal V}_2]]+[{\mathcal
V}_2,[Y^V,{\mathcal V}_1]].\end{align*} By the above induction
result we know that both summands in the last equality are of type
\ref{VF4}. Then we are done.

Consider now the induction step. Assume it is true for $s\geq 3$
that the vector fields in $[Y^V,{\mathcal C}]$ can be rewritten as
linear combination of vector fields of type \ref{VF4}, let us prove
it for $s+1$.
\begin{align*}[Y^V,&\, {\mathcal V}_{i_1\dots i_{s+1}}]=\\& \,-[{\mathcal V}_{i_1\dots i_{s}},[Y^V,{\mathcal V}_{i_{s+1}}]]
-[{\mathcal V}_{i_1\dots i_{s+1}},Y^V]\\ &\, = -[{\mathcal
V}_{i_1\dots i_{s}},[Y^V,{\mathcal V}_{i_{s+1}}]] +[{\mathcal
V}_{i_{s+1}},[Y^V,{\mathcal V}_{i_1\dots i_{s}}]].\end{align*} Both
summands are of type \ref{VF4} because $[Y^V,{\mathcal V}_{s}]$, $[Y^V,{\mathcal V}_{i_1\dots i_{s}}]$ are
of type~\ref{VF4} by induction. Now the proof by induction is concluded.
\end{itemize}
\end{proof}
\end{thm}

Immediate from this result, we have the following description of the
accessibility algebra at some points with non-zero velocity.
\begin{corol}
If $v_q\in {\rm Sym}^{(\infty)}(\mathscr Y)$, then
\[{\rm Lie}^{(\infty)}(Z,\mathscr{Y}^V)_{v_q}\subseteq T_{v_q} {\rm Sym}^{(\infty)}(\mathscr
Y).\] \label{Corol-Subset}
\end{corol}

\begin{proof}
Note that
\[{\rm Lie}^{(\infty)}(Z,\mathscr{Y}^V)_{v_q}={\rm
Lie}^{(\infty)}(Z,({\rm Sym}^{(\infty)}(\mathscr{Y}))^V)_{v_q},\]
because
\[\mathscr{Y}_q\subseteq  {\rm Sym}^{(\infty)}(\mathscr{Y})_q\] and \[{\rm Sym}^{(\infty)}(\mathscr{Y})^V_{v_q}\subseteq  {\rm
Lie}^{(\infty)}(Z,\mathscr{Y}^V)_{v_q}\] since $\langle Y_a\colon
Y_b \rangle^V=[Y_a^V,[Z,Y_b^V]]$.

As ${\rm Sym}^{(\infty)}(\mathscr Y)$ is geodesically invariant by
definition, the geodesic spray is tangent to ${\rm
Sym}^{(\infty)}(\mathscr Y)$, cf. Section~\ref{Sec:NablaRestricted}\@. Thus, on ${\rm Sym}^{(\infty)}(\mathscr
Y)$ all the vector fields of type \ref{VF3} in Theorem
\ref{Prop-Type-VF} are also tangent to ${\rm
Sym}^{(\infty)}(\mathscr Y)$. Then all the vector fields of type
\ref{VF4} in Theorem \ref{Prop-Type-VF} are also tangent to ${\rm
Sym}^{(\infty)}(\mathscr Y)$ because of the definition of Lie
bracket.
\end{proof}

\section{Description of the accessibility algebra} \label{Sec:Description}

We are going to introduce here how to compute inductively and explicitly the vector
fields in ${\rm Lie}^{\infty}(Z, {\mathscr{Y}}^V)$ since they would
be necessary later on. At zero velocity there are many objects that
vanish as described in~\cite{LM97}, but at non-zero velocity the
computations become more involved and it is useful to consider
objects defined along the projection map $\tau_Q \colon TQ
\rightarrow Q$ as described in~\cite{Eduardo1,Eduardo2}.

Given a connection $\nabla$ on $Q$, at any $v_q\in TQ$ there exists
a pointwise splitting of $T_{v_q}TQ$ analogous to the one in
(\ref{Eq:Split-0q}).
\begin{equation}\label{Eq-Split-vq} T_{v_q}TQ \simeq H_{v_q}(TQ)\oplus
V_{v_q}(TQ)\simeq T_q Q \oplus T_qQ.\end{equation} For any vector
field on $TQ$, the horizontal and the vertical projector at $v_q\in TQ$ are denoted
by ${\rm hor}_{v_q}\colon T_{v_q}TQ \rightarrow H_{v_q}(TQ)\simeq
T_qQ$, ${\rm ver}_{v_q}\colon T_{v_q}TQ \rightarrow V_{v_q}(TQ)
\simeq T_qQ$, respectively.

Remember that the \textbf{vertical lift} of a vector field $X$ on
$Q$ is the vector field $X^V$ on $TQ$ defined by $X^V(v_q)={\rm
vlft}_{v_q}(X(q))$, where ${\rm vlft}_{v_q}\colon T_qQ \rightarrow
T_{v_q}TQ$ is given by \begin{equation*}{\rm
vlft}_{v_q}(X_q)=\ds{\left.\frac{{\rm d}}{{\rm d}t}\right|_{t=0}
(v_q+tX_q)}\end{equation*} for any $v_q\in TQ$.

The \textbf{horizontal lift} of the tangent vector $v_q\in T_qQ$ at
point $q\in Q$ is the tangent vector in $T_{v_q}TQ$ given by
\begin{equation*} {\rm
hlft}_q(v_q)=(T_{v_q}\tau_Q|H_{v_q}TQ)^{-1}(v_q).\end{equation*}
Then, the \textbf{horizontal lift of the vector field} $X\in
\Gamma^\omega(TQ)$ is the vector field $X^H$ in $\Gamma^\omega(TTQ)$
defined by $X^H(v_q)={\rm hlft}_q(X(\tau_Q(v_q)))$.

Locally, if $X=X^i\partial /
\partial q^i$, then
\begin{align*}\ds{X^H(v_q)}&=\ds{X^i(q)\left(\frac{\partial}{\partial
q^i}-\Gamma^j_{ik}(q)v^k \frac{\partial}{\partial v^j}\right),}\\
\ds{X^V(v_q)}&=\ds{X^i(q)\frac{\partial}{\partial v^i}}.\end{align*}

It is possible to rewrite Lie brackets in ${\rm Lie}^{(\infty)}(Z,\mathscr{Y}^V)$
in terms of the splitting in (\ref{Eq-Split-vq}). The computation of
the vector fields in ${\rm Lie}^{(\infty)}(Z,\mathscr{Y}^V)$ can be
restricted to the brackets of the form
\begin{equation*}
 [X_k,[X_{k-1},[\dots,[X_2,X_1]\dots]]]
\end{equation*}
cf. \cite[Proposition 3.1]{LM97}. For vector fields $X$ and $W$ on $Q$, it can be computed that
\begin{align} [Z,X^H+W^V]&=(\nabla_vX-W)\oplus(-R(X,v)v+\nabla_vW), \label{Eq-Zbracket}\\
[Y^V,X^H+W^V]&=0\oplus (-\nabla_XY) \label{Eq-YVbracket},
\end{align} where $R$ is the curvature tensor associated with
$\nabla$ and $v$ denotes the vector field ${\rm Id}\colon TQ \rightarrow TQ$ defined along the projection $\tau_Q\colon TQ\rightarrow Q$.

The horizontal and vertical lift are defined analogously for vector
fields $X$ defined along the projection $\tau_Q\colon TQ\rightarrow
Q$, that is, $\tau_Q\circ X=\tau_Q$.  The vector fields with the
smallest length in ${\rm Lie}^{(\infty)}(Z,\mathscr{Y}^V)$ can be rewritten as follows
\begin{align} [Z,Y^V]&=-Y\oplus \nabla_vY, \label{Eq-ZYV}\\ [Y_1^V,[Z,Y_2^V]]&=0\oplus \langle Y_1 \colon Y_2 \rangle
\nonumber, \\
[Z,[Z,Y^V]]&=-2\nabla_vY \nonumber \\ & \oplus (R(Y,v)v+{\rm
ver}(\nabla^H_{v^H}(\nabla_vY)^V)) \label{Eq-ZZYV},
\end{align} where $\nabla^H$ denotes the horizontal lift of a
connection defined on \cite{Yano}. Note that for coordinates
$(q^i,v^i)$ the only non-zero Christoffel symbols
${}^H\Gamma^j_{ji}$ of $\nabla^H$ needed for our computations are:
\begin{equation}\label{eq:Gamma-Nabla-H}
 {}^H\Gamma_{j\overline{i}}^{\overline{l}}={}^H\Gamma_{\overline{j}i}^{\overline{l}}=\Gamma^l_{ji},
\end{equation}
where $i$ is the index corresponding to $q^i$ and $\overline{i}$ is
the index corresponding to $v^i$. From here it easily follows
$\nabla^H_{X^H}Y^V(v_q) \in V_{v_q}TQ$.

Note that in (\ref{Eq-ZYV}) and (\ref{Eq-ZZYV}) there are vector
fields depending on the velocities. To compute the iterated Lie
bracket from here we must use
\begin{align} [Z,X^H+W^V](v_q)&=({\rm ver}_{v_q}(\nabla^H_{v^H}X^V)-W) \nonumber \\&\oplus(-R(X,v)v+{\rm ver}_{v_q}(\nabla^H_{v^H}W^V)),
\label{Eq-ZbracketVeloc}\\
[Y^V,X^H+W^V](v_q)&={\rm ver}_{v_q}(\nabla^H_{Y^V}X^V) \nonumber \\&\oplus
({\rm ver}_{v_q}(\nabla^H_{Y^V}W^V+\nabla^H_{X^H}Y^V))
\label{Eq-YVbracketVeloc}.
\end{align}
instead of (\ref{Eq-Zbracket}) and (\ref{Eq-YVbracket}).
Observe that in~\eqref{Eq-ZbracketVeloc} and~\eqref{Eq-YVbracketVeloc} the vertical
projector is always acting over vertical vector fields because of
the definition of the horizontal lift $\nabla^H$ of $\nabla$.

\subsection{On the tangent space to a distribution}\label{Sub:TangentD}

It would be very useful to have a decomposition of the tangent space
to a distribution in terms of the horizontal and vertical subspace.
However this is only possible at points with zero velocity or if the
connection restricts to the distribution, cf.
Section~\ref{Sec:NablaRestricted}, as proved in the following
proposition.

\begin{prop}
If $\nabla$ restricts to a distribution $\mathcal{D}$, then
\begin{align*}
{\rm T}_{v_q}\mathcal{D} \cap H_{v_q}(TQ)  & \simeq \, T_qQ,\\
{\rm T}_{v_q}\mathcal{D}
\cap V_{v_q}(TQ)& \simeq \, \mathcal{D}_q, 
\end{align*}
for every $v_q\in \mathcal{D}$.
\label{Prop:TD-ifDrestrics}

\begin{proof}
%
Let $k$ be the rank of the distribution $\mathcal{D}$, consider a
$\mathcal{D}$-adapted basis $\{Y_1,\dots,Y_n\}$ on $Q$ such that
\begin{equation*}
 {\rm span}_{\mathbb{R}}\{ Y_1(q),\dots, Y_k(q)\}=\mathcal{D}_q,\quad \mbox{for every } q\in Q.
\end{equation*}

The Christoffel symbols associated with the $\mathcal{D}$-adapted basis $\{Y_1,\dots,Y_n\}$
are given by
\begin{equation*}
 \nabla_{Y_i}Y_j=\Gamma^l_{ij}Y_l.
\end{equation*}

As $\nabla$ restricts to $\mathcal{D}$, we have
\begin{equation*}
 \Gamma_{ia}^\alpha=0
\end{equation*}
for $\alpha=k+1,\dots,n$; $a=1,\dots,k$; $i=1,\dots,n$.
Then the basis of the horizontal subspace of $T_{v_q}TQ$ on $\mathcal{D}$ is tangent to $\mathcal{D}$ since
\begin{equation*}
 Y_i^H=Y_i-\Gamma_{ij}^lv^jY_l^V=Y_i-\Gamma_{ib}^cv^bY_c^V,
\end{equation*}
for $i=1,\dots,n$; $b,c=1,\dots,k$.

The tangency of vector fields on $TQ$ can be intrinsically characterized as follows. We let $\pi_{\mathcal{D}}\colon
TQ\rightarrow TQ/\mathcal{D}$ be the canonical projection such that the following diagram commutes
\begin{equation*}
\xymatrix{{TQ}\ar[rr]^{\pi_{\mathcal{D}}}\ar[rd]_{\tau_Q}&&
TQ/\mathcal{D}\ar[ld]^{\tau_{\mathcal{D}}}\\&Q&}
\end{equation*}

At $0_q\in TQ/\mathcal{D}$\@, there exists the following natural splitting
\begin{equation}\label{eq:Split-0}
T_{0_q}TQ/\mathcal{D}\simeq T_qQ \oplus (T_qQ/\mathcal{D}_q),
\end{equation}
cf.~\cite[Lemma~6.33]{AndrewBook}\@.  Hence we define the projection
$\pi_2\colon T_{0_q}TQ/\mathcal{D} \rightarrow T_qQ/\mathcal{D}_q$ onto the
second component of the splitting in~\eqref{eq:Split-0}\@.  
The vector field $Y_i^H$ on $TQ$ is tangent to $\mathcal{D}$
if and only if
\begin{equation*}
\pi_2((T_{v_q}\pi_{\mathcal{D}}\circ Y_i^H)(v_q))=0_q
\end{equation*}
for every\/ $v_q\in\mathcal{D}$\@. This equation is true because it is well-known that
$Y_c^V$ is tangent to $\mathcal{D}$ for $c=1,\dots,k$. The result follows from here.
%
\end{proof}
\end{prop}

Now, under suitable assumptions, we are going to describe
intrinsically the distribution of vector fields of type~\ref{VF3} in
Theorem~\ref{Prop-Type-VF}. First we need the following result.

\begin{prop}
If $\nabla$ restricts to a distribution $\mathcal{D}$, then
$\nabla^H$ restricts to the distribution $\mathcal{D}^V$.
\label{Prop:NablaHRestrictD}

\begin{proof}
It follows from~\eqref{eq:Gamma-Nabla-H} and from considering a
$\mathcal{D}$-adapted basis on $Q$. To be more precise, for any
section $C^aY_a^V$ of $\mathcal{D}^V$
\begin{align*}
 \nabla^H_{A^iY_i+B^iY_i^V}(C^aY_a^V)&= \, (A^iY_i+B^iY_i^V)(C^a)Y_a^V\\ &+\, \Gamma^a_{ic}A^iC_cY_a^V \in \Gamma^\infty(\mathcal{D}^V).
\end{align*}~\end{proof}
\end{prop}

\begin{prop}
 If $\nabla$ restricts to a distribution $\mathcal{D}$, then the smallest
distribution $\mathcal{D}^Z$ invariant under the geodesic spray
that contains $\mathcal{D}$ satisfies
\[\mathcal{D}^Z(v_q)\simeq \mathcal{D}_q\oplus \mathcal{D}_q\]
for every $v_q \in \mathcal{D}$.
\label{Corol-Zinv}

\begin{proof}
Let us prove it by induction on the number $s$ of iterated Lie brackets with the geodesic spray.

For $s=1$, see~\eqref{Eq-ZYV}. As $\nabla$ restricts to $\mathcal{D}$, the vertical
component of~\eqref{Eq-ZYV} sits in $\mathcal{D}$.

For $s=2$, see~\eqref{Eq-ZZYV}. As $\nabla$ restricts to
$\mathcal{D}$, the curvature tensor leaves $\mathcal{D}$ invariant,
cf. Proposition~\ref{Prop:Andrew98R}. Using also
Proposition~\ref{Prop:NablaHRestrictD} we have $[Z,[Z,Y^V]](v_q)\in
\mathcal{D}_q\oplus \mathcal{D}_q$ for every $v_q\in \mathcal{D}$.

Assume the result is true for $s$, let us prove it for $s+1$. Consider the following Lie bracket
\begin{align*}
 [Z,X^H+W^V](v_q)&=\, (\rm{ver}_{v_q}(\nabla^H_{v^H}X^V)-W) \\ &\oplus \, (-R(X,v)v+\rm{ver}_{v_q}(\nabla^H_{v^H}W^V)),
\end{align*}
where $X|_{\mathcal{D}}$, $W|_{\mathcal{D}}$ are
$\mathcal{D}$-valued vector fields along $\tau_Q$ and $\rm{ver}_{v_q}$ is
the projection onto the vertical subspace according to the splitting
in~\eqref{Eq-Split-vq}. Thus the vertical lifts of $X$ and $W$ are
sections of $\mathcal{D}^V$. By
Proposition~\ref{Prop:NablaHRestrictD} we conclude that both
$\nabla^H_{v^H}X^V$ and $\nabla^H_{v^H}W^V$ are vector fields
tangent to $\mathcal{D}$. Moreover they are vertical vector fields,
then their vertical projections are sections of
$\mathcal{D}$. As the curvature tensor leaves $\mathcal{D}$
invariant, cf. Proposition~\ref{Prop:Andrew98R}, the result follows.
\end{proof}

\end{prop}

\begin{prop}
If $\nabla$ restricts to ${\rm Sym}^{(\infty)}(\mathscr Y)$,
then
\begin{align}
{\rm Sym}^{(\infty)}(\mathscr Y)_q & \simeq \, {\rm Lie}^{(\infty)}(Z,\mathscr{Y}^V)_{v_q} \cap V_{v_q}(TQ),
\label{eq:AccAlgVert}\\
{\rm Lie}^{(\infty)}({\rm Sym}^{(\infty)}(\mathscr Y))_q & \subseteq \, {\rm Lie}^{(\infty)}(Z,\mathscr{Y}^V)_{v_q}
\cap H_{v_q}(TQ), \label{eq:AccAlgHor}
\end{align}
for every $v_q\in {\rm Sym}^{(\infty)}(\mathscr Y)$.
\label{Corol-AccDistrNonZero}

\begin{proof}
Remember that
\begin{equation*} {\rm Lie}^{(\infty)}(Z,\mathscr{Y}^V)_{v_q}= {\rm Lie}^{(\infty)}(Z,{\rm Sym}^{(\infty)}(\mathscr
Y)^V)_{v_q}.
\end{equation*}
The hypothesis, Corollary~\ref{Corol-Subset}  and Proposition~\ref{Prop:TD-ifDrestrics}
leads to
\begin{align*}
{\rm Lie}^{(\infty)}(Z,{\rm Sym}^{(\infty)}(\mathscr Y)^V)_{v_q}& \subseteq \, T_{v_q}{\rm Sym}^{(\infty)}(\mathscr Y)\\
& \simeq \,
T_qQ \oplus {\rm Sym}^{(\infty)}(\mathscr Y)_q.
\end{align*}
Since ${\rm Sym}^{(\infty)}(\mathscr Y)^V_{v_q}$ sits in ${\rm Lie}^{(\infty)}(Z,{\rm Sym}^{(\infty)}(\mathscr Y)^V)_{v_q}$
for every $v_q\in TQ$,
\begin{equation*}
 {\rm Sym}^{(\infty)}(\mathscr Y)_q  \simeq  {\rm Lie}^{(\infty)}(Z,\mathscr{Y}^V)_{v_q} \cap V_{v_q}(TQ).
\end{equation*}

To prove~\eqref{eq:AccAlgHor} first note that by Theorem~\ref{Prop-Type-VF} and Proposition~\ref{Corol-Zinv}
\begin{equation*}
 {\rm Sym}^{(\infty)}(\mathscr Y)_q  \subseteq  {\rm Lie}^{(\infty)}(Z,\mathscr{Y}^V)_{v_q} \cap H_{v_q}(TQ).
\end{equation*}
Using~\eqref{Eq-ZYV} and~\eqref{Eq-Zbracket}, for every $Y_1$, $Y_2\in \Gamma^\omega({\rm Sym}^{(\infty)}(\mathscr Y))$ we have
\begin{equation*}
 [[Z,Y_1^V],[Z,Y_2^V]](v_q)=[Y_1,Y_2]_q\oplus W_q,
\end{equation*}
for every $v_q\in {\rm Sym}^{(\infty)}(\mathscr Y)$ where $W\in \Gamma^\omega({\rm Sym}^{(\infty)}(\mathscr Y))$ because
of~\eqref{eq:AccAlgVert}. From here~\eqref{eq:AccAlgHor} follows.
\end{proof}

\end{prop}

\begin{thm} Let $\Sigma$ be an analytic ACCS. If $\nabla$ restricts to ${\rm Sym}^{(\infty)}(\mathscr
Y)$,
then
\begin{enumerate}
\item $\Sigma$ is
accessible from $v_q\in {\rm Sym}^{(\infty)}(\mathscr Y)$
if and only if \begin{align*}{\rm Sym}^{(\infty)}(\mathscr Y)_q& =
\, T_qQ, \\ {\rm Lie}^{(\infty)}(Z,\mathscr{Y}^V)_{v_q}&= \, T_{v_q}
{\rm Sym}^{(\infty)}(\mathscr Y).\end{align*}
\item $\Sigma$ is configuration
accessible from $v_q\in {\rm Sym}^{(\infty)}(\mathscr Y)$ if and
only if
\[{\rm Lie}^{(\infty)}(Z,\mathscr{Y}^V)_{v_q}= T_{v_q} {\rm Sym}^{(\infty)}(\mathscr
Y).\]
\end{enumerate}
\begin{proof} The proof follows from the proof of Theorem~\ref{Th:SussmaanEtAl} and Theorem~\ref{Th:LM97} using~\eqref{eq:AccAlgVert} in
Proposition~\ref{Corol-AccDistrNonZero}. %

Moreover, if ${\rm Lie}^{(\infty)}_{v_q}(Z,\mathscr{Y}^V)= T_{v_q}
{\rm Sym}^{(\infty)}(\mathscr Y)$, then horizontal subspace of ${\rm
Lie}^{(\infty)}_{v_q}(Z,\mathscr{Y}^V)$ at $v_q\in {\rm
Sym}^{(\infty)}(\mathscr Y)$ is the entire tangent space $T_qQ$.
\end{proof}
\label{Thm-AccDistrNonZero}
\end{thm}

\section{Examples}\label{Sec:Examples}

Let us use Theorem~\ref{Thm-AccDistrNonZero} to describe the
accessibility of some mechanical controls systems defined by an affine connection from points with
nonzero velocity. These examples are more thoroughly described
in~\cite[Section 7.4.2-3]{AndrewBook}.

\subsection{Planar body with variable-direction thruster}

The configuration manifold for the system is $Q=\mathbb{S}^1\times
\mathbb{R}^2$. We denote the coordinates by $(\theta,x,y)$. The
Riemannian metric for the system is
\begin{equation*}
\mathbb{G}=J {\rm d}\theta \otimes {\rm d}\theta +m({\rm d}x \otimes
{\rm d}x +{\rm d}y \otimes {\rm d}y),
\end{equation*}
where $m$ is the mass of the body and $J$ is its moment of inertia
about its center of mass. Then the Christoffel symbols are all zero.

This systems has two input vector fields given by
\begin{align*}
Y_1&= \, \frac{\cos \theta}{m} \frac{\partial}{\partial x}
+\frac{\sin \theta}{m} \frac{\partial}{\partial y},\\
Y_2&=\, -\frac{h}{J}\frac{\partial}{\partial \theta}-\frac{\sin
\theta}{m} \frac{\partial}{\partial x}+\frac{\cos \theta}{m}
\frac{\partial}{\partial y}.
\end{align*}
This system is an affine connection control system and we can study
its accessibility from nonzero velocities in ${\rm
Sym}^{(\infty)}\mathscr{Y}$. As ${\rm
Sym}^{(\infty)}\mathscr{Y}_q=T_qQ$, it is trivial that $\nabla$ 
restricts to ${\rm Sym}^{(\infty)}\mathscr{Y}_q$. Moreover, ${\rm
Lie}^{(\infty)}({\rm Sym}^{(\infty)}\mathscr{Y})_q\simeq T_qQ$.
Having in mind Proposition~\ref{Prop:TD-ifDrestrics},
Proposition~\ref{Corol-AccDistrNonZero} and
Theorem~\ref{Thm-AccDistrNonZero}, we can conclude that the system
is accessible and configuration accessible from $v_q\in {\rm
Sym}^{(\infty)}\mathscr{Y}$.

\subsection{Rolling disk}

The configuration manifold is $Q=\mathbb{R}^2\times \mathbb{S}^1
\times \mathbb{S}^1$. The coordinates are denoted by
$(x,y,\theta,\phi)$. This system has constraints and to study the
accessibility of the system the constrained connection
in~\cite[Section 4.5.5]{AndrewBook} will have to be considered. The
configuration manifold then is restricted to $\overline{Q}$ with
dimension 2.

All the Christoffel symbols associated with the constrained
connection are zero in terms of the $\mathbb{G}$-orthogonal generators
\begin{align*}
X_1&=\; \rho \cos \theta \frac{\partial}{\partial x}+\rho \sin
\theta \frac{\partial}{\partial y} +\frac{\partial}{\partial
\phi},\\
X_2&=\;\frac{\partial}{\partial \theta}.
\end{align*}

Two input vector fields are considered for this system
\begin{equation*}
Y_1\,=\, \frac{1}{J_{\rm spin}} X_2, \qquad Y_2\,=\,
\frac{1}{m\rho^2+J_{\rm roll}}X_1.
\end{equation*}

Let us consider different cases of input vector fields:
\begin{enumerate}
\item $Y_1$ only: The constrained affine connection restricts to
${\rm Sym}^{(\infty)}Y_1=Y_1$. The system is not accessible from
$v_q\in {\rm Sym}^{(\infty)}Y_1$ because of
Theorem~\ref{Thm-AccDistrNonZero}.

\item $Y_2$ only: The constrained affine connection restricts to
${\rm Sym}^{(\infty)}Y_2=Y_2$. The system is not accessible from
$v_q\in {\rm Sym}^{(\infty)}Y_2$ because of
Theorem~\ref{Thm-AccDistrNonZero}.

\item $Y_1$ and $Y_2$: The constrained affine connection restricts to
${\rm Sym}^{(\infty)}(Y_1,Y_2)_q=T_q\overline{Q}$. Using similar
reasoning as for the planar rigid body, we conclude that the system
is accessible and configuration accessible from $v_q\in {\rm
Sym}^{(\infty)}\mathscr{Y}$.
\end{enumerate}

In order to decide about the configuration accessibility for
single-input case, we will have to compute ${\rm
Lie}^{(\infty)}(Z,\mathscr{Y})$. In these particular cases
computations show that the horizontal subspace of ${\rm
Lie}^{(\infty)}(Z,\mathscr{Y}^V)$ is never the entire tangent space
$T_qQ$. Thus the system is not configuration accessible from $v_q\in
{\rm Sym}^{(\infty)}\mathscr{Y}$ when $\mathscr{Y}=\{Y_1\}$ or
$\mathscr{Y}=\{Y_2\}$.


\section{CONCLUSIONS AND FUTURE WORKS}

\subsection{Conclusions}

In this paper we have generalized results about the accessibility
for mechanical control systems called affine connection control
systems that were proved in \cite{LM97}\@. Precisely,
Theorem~\ref{Th:LM97} has been extended to Theorem~\ref{Thm-AccDistrNonZero}. The latter result characterizes intrinsically the
accessibility properties of a mechanical control system at points
with velocities in ${\rm Sym}^{(\infty)}\mathscr{Y}$ if
$\nabla$ restricts to ${\rm Sym}^{(\infty)}\mathscr{Y}$.

First we have identified the primitive brackets at nonzero velocity
in Theorem \ref{Prop-Type-VF}. This result and Proposition~\ref{Prop:TD-ifDrestrics} have allowed us to
characterize the accessibility distribution at points in ${\rm
Sym}^{(\infty)}\mathscr{Y}$ as long as $\nabla$ restricts to
${\rm Sym}^{(\infty)}\mathscr{Y}$, see Corollary~\ref{Corol-Subset},
Propositions~\ref{Corol-Zinv} and~\ref{Corol-AccDistrNonZero}.

\subsection{Future Works}

It still remains to extend Theorem \ref{Thm-AccDistrNonZero} to any
velocity point under the assumption of having $\nabla$ restricted to
${\rm Sym}^{(\infty)}\mathscr{Y}$ and also discarding this
assumption. Once we succeed to completely characterize the accessibility
algebra, our efforts will focus on extending the characterization of
controllability in the sense of STLC given in~\cite{LM97} at points with
non-zero velocity, cf. Definition~\ref{Def:AccControl},

It is expected that infinitesimal holonomy algebra, cf.~\cite{KN},
will play an important role for these extensions.


\addtolength{\textheight}{-2cm}   

\section{ACKNOWLEDGMENTS}

The author gratefully acknowledges Professor A. D. Lewis for useful discussions.

\end{document}